\documentclass[12pt]{amsart}
\usepackage{amssymb}
\usepackage{yhmath}

\input{epsf}
\input xy  \xyoption{all}
\usepackage[pdftex]{graphicx} \usepackage{epstopdf}
\title{ Real solutions of a problem in enumerative geometry}

\author{L\'aszl\'o M. Feh\'er}
\address{Department of Analysis, E\"otv\"os University, Budapest, Hungary}
\email{lfeher@renyi.mta.hu}

\author{\'Akos K.\ Matszangosz}
\address{Department of Analysis, E\"otv\"os University, Budapest, Hungary}
\email{matszangosz.akos@gmail.com}

\thanks{The first named author is supported by the OTKA grants 72537 and 81203.}

\newtheorem{fact}{Fact}[section]
\newtheorem{lemma}[fact]{Lemma}
\newtheorem{theorem}[fact]{Theorem}
\newtheorem{defi}[fact]{Definition}
\newtheorem{exa}[fact]{Example}
\newtheorem{exas}[fact]{Examples}
\newtheorem{rremark}[fact]{Remark}
\newtheorem{proposition}[fact]{Proposition}
\newtheorem{corollary}[fact]{Corollary}
\newenvironment{remark}{\begin{rremark} \rm}{\end{rremark}}
\newenvironment{definition}{\begin{defi} \rm}{\end{defi}}

\newenvironment{examples}{\begin{exas} \rm}{\end{exas}}

\def\semmi#1{}

\newcommand{\smx}[4]{\bigl(\begin{smallmatrix}{#1}&{#2}\\{#3}&{#4}\end{smallmatrix}\bigr)}

\DeclareMathOperator{\sol}{Sol}

\DeclareMathOperator{\C}{\mathbf C}
\DeclareMathOperator{\Z}{\mathbf Z}

\DeclareMathOperator{\R}{\mathbf R}

\DeclareMathOperator{\Q}{\mathbf Q}

\renewcommand{\P}{\mathbf P}

\DeclareMathOperator{\Hom}{Hom}

\DeclareMathOperator{\gr}{Gr}
\DeclareMathOperator{\Gr}{Gr}
\DeclareMathOperator{\res}{Res}

\DeclareMathOperator{\sym}{Sym}

\DeclareMathOperator{\Stab}{Stab}
\DeclareMathOperator{\GL}{GL}

\DeclareMathOperator\im{Im}
\def\iso{\cong}
\def\ds{\displaystyle}
\newcommand{\bigtimes}{\mathop{\textsf{X}}}
\newcommand{\twocases}[7][1mm]{{#2}\begin{cases} {#3}& \hbox{\rm if}\ \ {#4} \\[{#1}] {#5}& \hbox{\rm #6}\ \ {#7}\end{cases}}

\begin{document}

\begin{abstract}
We study a 2-parameter family of enumerative problems over the reals. Over the complex field, these problems can be solved by Schubert calculus. In the real case the number of solutions can be different on the distinct connected components of the configuration space, resulting in a solution function. The cohomology calculation in the real case only gives the signed sum of the solutions, therefore in general it only gives a lower bound on the range of the solution function. We calculate the solution function for the 2-parameter family and we show that in the even cases the solution function is constant modulo 4. We show how to determine the sign of a solution and describe the connected components of the configuration space. We translate the problem to the language of quivers and also give a geometric interpretation of the sign. Finally, we discuss what aspects might be considered when solving other real enumerative problems.
\end{abstract}
\maketitle
 \tableofcontents

\section{Introduction}
In this paper we study a 2-parameter family of enumerative problems over the reals:

\emph{Given four $b$-dimensional subspaces in generic position in a $2b$-dimesional space what is the number of $2a$-dimensional subspaces having $a$-dimensional intersection with all the given subspaces?}

These are so called \emph{Schubert problems} which can be solved by Schubert calculus over the complex field. In fact there is a more elegant way to calculate, see \cite{vakil:sch-induction}. In general the answer to an enumerative problem over the complex field is a number. Over the reals, however the subvariety of nongeneric configurations is a real hypersurface with possibly non-connected complement. Consequently the number of solutions depend on the \emph{chamber} (connected component) of the space of configurations, so instead of a number we have a \emph{solution function}, an integer-valued function on the chambers. The minimal information we want to know is the range of the solution function. These solution functions recently were studied intensively see e.g. \cite{finashin_abundance_2012}, \cite{okonek-teleman}, \cite{sottile_frontiers_2010}, \cite{sottile4}. The case of $a=2$ was calculated in \cite[Thm 8.1]{sopsot}

Our solution is fairly elementary, in fact we adopt the method of R. Vakil in \cite{vakil:sch-induction} from the complex case to the real (We learned the idea from Tam\' as Terpai, and only later found the paper of R. Vakil). We present the method in Section \ref{sec:number}.

In Section \ref{sec:schubert} we explain in what extent Schubert calculus can be extended to the real case: The cohomology calculation gives us the \emph{signed sum of the solutions}, therefore in general it only gives a lower bound of the solution function. It turns out that in our cases this lower bound is sharp.

Section \ref{sec:sign} is devoted to the determination of the sign of a solution. The sign decides whether 2 orientations of a certain vector space coincide so we have to calculate the sign of the determinant of the  change of basis. It turns out that this determinant is a resultant, similarly to the case of lines on a cubic surface studied in \cite{okonek-teleman} and \cite{finashin_abundance_2012}.

The resultant leads to a combinatorial description of the sign which is presented in Section \ref{sec:mod4}. In this section we also show that for $a$ and $b$ even, the number of solutions is constant modulo 4. This phenomenon was shown in \cite{sottile4} and \cite{heinhillarsottile} for a class similar problems. They study so called osculating solutions. It can be considered as an other ``real form'' for the complex Schubert problems, so we can expect similar behavior. The modulo 4 property also holds for the problem of lines on a cubic surface. This property was studied in \cite{benedetti_spin_1995}. It would be interesting to find a unified explanation. We also show that if $a$ and $b$ are both divisible by $2^k$, then the number of solutions is constant modulo $2^{k+1}$.

In Section \ref{sec:chambers} we list the chambers of the generic configurations.

In Section \ref{sec:quiver} we explain the translation to the language of quivers.

In Section \ref{sec:cross} we explain the geometric meaning of the sign of a solution in terms of cross ratios. We shortly mention the case of unordered four-tuples of subspaces.

In Section \ref{sec:what} we speculate on what could be considered as a solution to a real enumerative problem and review the problem of lines on a cubic surface.

We would like to thank Tam\'as Terpai, whose explanation initiated the paper, Rich\'ard Rim\'anyi for many conversations on the topic, Frank Sottile for important comments on the first version of the manuscript and M\'aty\'as Domokos for educating us on quivers. While finishing the manuscript the first author was enjoying the hospitality of the Alfr\'ed R\'enyi Institute.

\section{The enumerative question: The number of balanced subspaces}\label{sec:number}
In this paper we only consider real and complex vector spaces. If not stated otherwise, we allow both cases.
\begin{definition} \label{de:generic} Suppose that four subspaces $V_1,\dots,V_4$  of dimension $b$ of the vector space $E$ of dimension $2b$ are given. We call a $2a$-dimensional subspace $W$ of $E$ \emph{balanced} relative to the subspaces $V_i$, if its intersection with the subspaces $V_i$ has dimension $a$ for all $i$. When the subspaces $V_i$ are fixed in advance, we simply say that $W$ is balanced.
\end{definition}
The classical problem of determining the number of lines intersecting given four lines can be rephrased as finding the number of balanced planes when the subspaces $V_i$ are two-dimensional planes in a four-dimensional vector space. Our aim is to solve this problem in the general case, in other words we want to determine the number of balanced subspaces in arbitrary dimensions.

To calculate the number of balanced subspaces we use a graph construction. We learned the idea from Tam\'as Terpai. Later we learned that this construction was used in the so-called four subspace problem before. The earliest occurrence we found is \cite{turnbull_projective_1942}. R.\ Vakil in \cite{vakil:sch-induction} sketches the construction more along the lines we do. For $2a=b$ he attributes the construction to H. Derksen.
\subsection{The graph construction}\label{sec:graph}
\begin{definition}
Let $A,B,C$ be subspaces of a real or complex vector space $V$ such that both $(A,B)$ and $(B,C)$ are complementary subspaces (in the algebraic sense): $A\cap B=B\cap C=\{0\}$ and $A+B=C+B=V$. By the direct sum property, for all $a\in A$ there is a unique $b\in B$ such that $a+b\in C$. This implies that there is a unique map
  \[\gamma=\gamma^{\scriptscriptstyle C\subset V}_{\scriptscriptstyle A\to B}:A\to B,\]
such that $\gamma(a)+a\in C$ for all $a\in A$.
\end{definition}
Since the subspaces are linear, it follows that $\gamma$ is linear, and its graph in $A+B=V$ is the subspace $C$. Another way to construct $\gamma$ is using the projections $\pi_A,\pi_B$ defined by the direct decomposition $V=A\oplus B$. By our assumption $\pi_A|_C$ is invertible and $\gamma=\pi_B(\pi_A|_C)^{-1}$. Notice that if additionally $A\cap  C=\{0\}$, then $\gamma^{\scriptscriptstyle C\subset V}_{\scriptscriptstyle B\to A}$ is also defined and equal to the inverse of $\gamma^{\scriptscriptstyle C\subset V}_{\scriptscriptstyle A\to B}$.\\

Returning to our problem, assume now that the subspaces $V_i$ are pairwise transversal. Then we can define the maps
\[\gamma_i:=\gamma^{\scriptscriptstyle V_i\subset E}_{\scriptscriptstyle V_1\to V_2}  \]
for $i=3$ and $4$.
\begin{proposition}\label{gamma3-4}If a  subspace $W\subset E$ is balanced then
\[\gamma_3(V_1\cap W)=\gamma_4(V_1\cap W)=V_2\cap W.\]
\end{proposition}
\begin{proof} Let us use the notation $W_i:=V_i\cap W$ for $i=1,2,3,4$. Then it is easy to see that
\[\gamma_i|_{W_1}=\gamma^{\scriptscriptstyle W_i\subset W}_{\scriptscriptstyle W_1\to W_2}  \]
for $i=3,4$. Therefore $\gamma_i(W_1)\subset  W_2$. By our previous remark  $\gamma^{\scriptscriptstyle W_i\subset W}_{\scriptscriptstyle W_1\to W_2}$ is invertible, so $\gamma_i(W_1)=W_2$.
\end{proof}
We call $\varphi:={\gamma_3}^{-1}\gamma_4$ the map corresponding to the configuration $V_1,V_2,V_3,V_4$.
\begin{corollary} \label{bijection}The map $W\mapsto W\cap V_1$ is a bijection between the balanced subspaces of $E$ and the $a$-dimensional invariant subspaces of the corresponding map $\varphi$.
\end{corollary}
\begin{proof}By Proposition \ref{gamma3-4}.,\ if $W$ is a $2a$-dimensional balanced  subspace then $W_1$ is an $a$-dimensional invariant subspace of $\varphi$. On the other hand if $W_1$ is an $a$-dimensional invariant subspace of $\varphi:={\gamma_4}^{-1}\gamma_3$, then $W:=W_1+ \gamma_3(W_1)$ is balanced, since
\[ V_1\cap W=W_1,\ V_2\cap W=\gamma_3(W_1),\ V_3\cap W=\{w+\gamma_3(w):w\in W_1\},\ V_4\cap W=\{w+\gamma_4(w):w\in W_1\}.\]
\end{proof}
\begin{remark}\label{rem:distinct} Since $\varphi $ is invertible, all the eigenvalues are non-zero. Similarly, since $V_3\cap V_4=\{0\}$, 1 is not an eigenvalue of $\varphi $.
\end{remark}
We are only interested in \emph{generic} configurations. A subset of the configuration space is a good choice for genericity, if it is open and the solution function is locally constant. Ideally we choose the maximal such subset. We will see later in Remark \ref{rem:generic}. that the following definition is good in  this maximal sense.

\begin{definition}  Four subspaces $V_1,\dots,V_4$  of dimension $b$ of the vector space $E$ of dimension $2b$ are in \emph{general position} (or \emph{generic}) if the corresponding map $\varphi:V_1\to V_1$ has $b$ distinct eigenvalues (over $\C$), none of them is equal to 0 or 1.
\end{definition}
For the complex case Corollary \ref{bijection}. immediately implies that
\begin{theorem}(R. Vakil \cite{vakil:sch-induction}) \label{c-solutions} Let the subspaces $V_1,\dots,V_4$  of dimension $b$ of the complex vector space $E$ of dimension $2b$ be in general position. Then the number of $2a$-dimensional balanced  subspaces is
\[N_{\C}(b,a)=\binom{b}{a}.\]
\end{theorem}
For the real case notice that the eigenvalues of $\varphi$ are either real or complex conjugate pairs. So the corresponding minimal invariant subspaces of $\varphi$ are one or two dimensional, respectively. Denoting the number of complex conjugate pairs by $c$ we get:
\begin{theorem} \label{r-solutions} Let the subspaces $V_1,\dots,V_4$  of dimension $b$ of the real vector space $E$ of dimension $2b$ be in general position. If the corresponding map $\varphi:V_1\to V_1$ has $c$ complex conjugate pairs among its eigenvalues, then the number of $2a$-dimensional balanced  subspaces is
\begin{equation}\label{eq:real-sols}
N_{\R}(b,a,c)=\sum\limits_{i=0}^c\binom{c}{i}\binom{b-2c}{a-2i}.
\end{equation}
\end{theorem}

\begin{remark} \label{reverse}Notice that we can construct four-tuples of spaces with a given $\varphi$: Let  $\varphi:F\to F$ be a linear map:
 \[ E=F\oplus F,\ V_1=F\oplus 0,\  V_2=0\oplus F, \ V_3=\{\big(f,f\big):f\in F\},\ V_4=\{\big(f,\varphi(f)\big):f\in F\}. \]
\end{remark}
As far as we know this is the only non-trivial infinite family of real enumerative problems where the values of the solution functions are known.

\begin{remark}\label{rem:sot} This can also be considered as the following problem: In how many ways can one factor a degree $b$ polynomial $\chi$ with real coefficients into two polynomials $\chi_C$ and $\chi_D$ with real coefficients of degree $a$ and $b-a$ respectively? (See discussion following Theorem \ref{thm:sign-alg-all}.) Such a factorization is a special case of the paper \cite{sopsot} of E.\ Soprunova and F.\ Sottile, where they examine the number of such factorizations, possibly to multiple polynomials. As they point out, the general case can be calculated by generating functions. The number of such factorizations in the currently examined case is obtained as the coefficient of $x_1^ax_2^{b-a}$ in the polynomial $(x_1+x_2)^{b-2c}(x_1^2+x_2^2)^c$.
\end{remark}

\begin{remark}\label{bounds} Since all real solutions are also complex, the number of complex solutions is a theoretical upper bound. For some real enumerative problems this cannot be obtained, like the number of flexes of a smooth plane curve \cite{klein}. In our case the maximum of $N_{\R}(b,a,c)$ is obtained for $c=0$, i.e. when all eigenvalues are real. $N_{\R}(b,a,0)=N_{\C}(b,a)$, i.e. the theoretical upper bound can be obtained. This also follows from a result of R. Vakil \cite[Prop. 1.3]{vakil:sch-induction} stating that all Schubert problems are enumerative over $\R$. It is also an immediate consequence of the Mukhin-Tarasov-Varchenko theorem in \cite{mukhin-tarasov-varchenko:shapiro} (earlier called Shapiro-Shapiro conjecture).

  The minimum is obtained when $c$ is maximal: For $u=\lfloor a/2\rfloor$ and $v=\lfloor b/2\rfloor=c$ the number of solutions is
\[ \twocases{N_{\R}(b,a,v)=}{\ 0}{\text{\rm $a$ odd and $b$ even}}
{\ds\binom{v}{u}}{\text{otherwise.}}{}  \]
Notice that for large $b$ there are large gaps among the possible values of the solution function  $f(c)=N_{\R}(b,a,c)$.
\end{remark}
\begin{examples} The example of $a=2,\ b=4$ is calculated in \cite[p.16]{garcia-puente_secant_2012}, where for the values $c=0,1,2$ we have $6,2,2$ solutions, respectively. In the next table we list the solution functions for $b=8$ and $a<8$. Notice the (obvious) symmetry.

\begin{table}\setlength{\tabcolsep}{20pt}

\begin{tabular}{|c|ccccc|}
\hline
$a\backslash c$	&	0 	&	1	&	2	&	3	&	4\\
\hline
1	&	8	&	6	&	4	&	2	&	0\\
2	&	28	&	16	&	8	&	4	&	4\\
3	&	56	&	26	&	12	&	6	&	0\\
4	&	70	&	30	&	14	&	6	&	6\\
5	&	56	&	26	&	12	&	6	&	0\\
6	&	28	&	16	&	8	&	4	&	4\\
7	&	8	&	6	&	4	&	2	&	0\\
\hline
\end{tabular}
\bigskip
\caption{The solution functions $N_{\R}(8,a,c)$.}
\end{table}
\end{examples}
\section{Schubert calculus} \label{sec:schubert}Results of this section will not be used in the rest of the paper. On the other hand we explain our motivation to assign signs to balanced subspaces.

The traditional method to solve enumerative problems over the complex numbers is Schubert calculus. Consider the subvariety
\[\sigma_V:=\{W\in\gr_{2a}(E):\dim(V\cap W)\geq a\}  \]
of the complex Grassmannian manifold $\gr_{2a}(E)$, where $V$ is a $b$ dimensional subspace of a $2b$-dimensional complex vector space $E$. Then assuming that the $b$-dimensional subspaces $V_1,\dots,V_4$ are generic, the number of $2a$-dimensional balanced  subspaces is the number of points of the 0-dimensional subvariety
\[\bigcap_{i=1}^4\sigma_{V_i}.\]
Using the fact that complex manifolds are canonically oriented, we can use cohomology:
\[  N_{\C}(b,a)=\int\limits_{\gr_{2a}(E)}[\sigma_V]^4, \]
where $[\sigma_V]\in H^*(\gr_{2a}(E),\Z)$ is the integer cohomology class represented by the variety $\sigma_V$. If we choose a complete flag $F_\bullet$ in $E$, such that $F_b=V$, then we can identify $\sigma_V$ with the Schubert variety $\sigma_\lambda=\sigma_\lambda(F_\bullet)$, where $\lambda$ is the partition with $\lambda_j=b-a$ for $j=1,\dots,a$. Now we can use the usual machinery of Schubert calculus. We warn the interested readers that the calculation is surprisingly complicated. However it is easy using puzzles or checkers (see \cite{Knutson-Tao-Woodward} and \cite{vakil:lr}).

Things get more complicated over the real number field as real subvarieties do not always represent  integer cohomology classes. One possibility is to use mod 2 cohomology. By the theorem of Borel and Haefliger all the calculations are the same, but in the end we only calculated the number of solutions modulo 2.

\begin{figure}
\begin{picture}(50,50)(2,2)
\thicklines
\put(0,50){\line(0,-1){40}}
\put(10,50){\line(0,-1){40}}
\put(20,50){\line(0,-1){40}}
\put(30,50){\line(0,-1){40}}
\put(40,50){\line(0,-1){40}}
\put(50,50){\line(0,-1){40}}
\put(0,50){\line(1,0){50}}
\put(0,40){\line(1,0){50}}
\put(0,30){\line(1,0){50}}
\put(0,20){\line(1,0){50}}
\put(0,10){\line(1,0){50}}

\put(53,50){\small $4$}
\put(-15,25){\normalsize $a$\Large$\bigg\lbrace$}
\put(0,10){\large$\underbrace{\,\qquad \,\quad \,}_{ b-a}$}
\end{picture}

\caption{The Young diagram corresponding to the enumerative problem}
\end{figure}
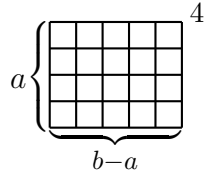

Assuming that $a$ and $b$ are even, by a result of Ehresmann \cite{ehresmann} the subvariety $\sigma_V$ will represent a cohomology class in $H^*(\gr_{2a}(E),\Z)$. We also have a limited version of the Schubert calculus (\cite{realschubert}) for rational cohomology which says that we have a (degree doubling) ring isomorphism
\[d:H^*(\gr_{u}(\C^{v}),\Q)   \to  H^*(\gr_{2u}(\R^{2v}),\Q),\]
such that $d([\sigma_\lambda])=[\sigma_{d\lambda}]$, where $d\lambda$ is the "double" of the partition $\lambda=(\lambda_1,\lambda_2\dots,\lambda_k)$:
\[ d\lambda=(2\lambda_1,2\lambda_1,2\lambda_2,2\lambda_2,\dots,2\lambda_k,2\lambda_k).  \]

Using the notation $a=2u,\ b=2v$ and $\lambda$ for the partition with $\lambda_j=v-u$ for $j=1,\dots,u$.
\[\int\limits_{\gr_{2a}(E)}[\sigma_V]^4=\int\limits_{\gr_{a}(\C^b)}[\sigma_\lambda]^4=N_{\C}(v,u)=\binom{v}{u}. \]
This number is only a lower bound for the number of balanced spaces, since in the real case we can have different orientations. Hence $\binom{v}{u}$ is actually the signed sum of the solutions. Looking at (\ref{eq:real-sols}) we can see that the lower bound is obtained when all eigenvalues of $\varphi$ come in complex conjugate pairs.

Similar interpretations can probably be given for the cases when $a$ or $b$ are odd using twisted cohomology (in the spirit of \cite{okonek-teleman}).

In the next section we explain how to determine the sign of a balanced subspace for the case when $a$ and $b$ are even. Concerning the other cases, see Remark \ref{rem:odd}.
\section{The sign of a balanced subspace}\label{sec:sign}
Following the cohomological interpretation of the previous section we fix orientations of the Grassmannian $\gr=\gr_{2a}(E)$ and the normal spaces of the (smooth part of the) subvarieties $\sigma_{V_i}$. If these subvarieties intersect transversally at a point $W$, then there are two possible orientations of the tangent space $T_W\gr$ (or, to be fancy, the normal space $N_W \{W\}\subset\gr$). The sign of $W$ is positive if the orientation of $T_W\gr$ agrees with  the orientation of $\bigoplus_{i=1}^4N_W\sigma_{V_i}$ under the canonical isomorphism of the two spaces, and negative otherwise.

Changing one of these orientations the sign changes as well, however we will show that there is a canonical choice for all orientations and that the sign has a (somewhat) geometric meaning.
\subsection{Calculating the normal spaces}
Again, fix a $2b$-dimensional vector space $E$, a $2a$-dimensional subspace $V\subset E$ and an $a$-dimensional subspace $W\subset V$. To calculate the normal spaces, we recall the following well known fact (see \cite{Kirillov} Corollary 2.19, page 13):

\begin{lemma}
\label{LemLie}
Let $G$ be a Lie group acting on the manifold $M$, and denote the orbit of $m\in M$ by $\mathcal{O}_m:=G\cdot m$. Then the tangent space $T_m\mathcal{O}_m$ is isomorphic to $T_e G/T_e H$, where $H=\Stab_G(m)$ is the stabilizer subgroup of $m$.
\end{lemma}

We also need the following lemma from linear algebra:
\begin{lemma}
\label{LemAlg}
Let $C$ and $D$ be finite-dimensional vector spaces, and  $C'\subset C$, $D'\subset D$ subspaces. Using the notation $A:=\Hom(C,D)$ and $A':=\left\{\psi\in A:\psi(C')\subset D'\right\}$, the relation
$$A/A'\cong \Hom(C', D/D')$$
holds.
\end{lemma}
\begin{proof}
Let $\pi:D\to D/D'$ be the canonical quotient map and $\iota: C'\to C$ the inclusion. Define $q:\Hom(C,D)\rightarrow \Hom(C',C/D')$ by \begin{equation}\label{de:q}
q(\psi):=\pi\circ \psi\circ\iota.
\end{equation}
Then clearly $q$ is surjective and $\ker(q)=A'$. By applying the classical isomorphism theorem, this gives
\begin{equation}
\label{isom}
A/\ker(q)\cong q(A)=\Hom(C', D/D'),
\end{equation}
which finishes the proof.
\end{proof}
For $\GL(E)$ acting on $\gr_{2a}(E)$ the lemmas imply the well known fact that $T_W\Gr_{2a}(E)\cong \Hom(W,E/W)$.

Now to calculate the normal space $N_W\sigma_{V}$ notice first that $\sigma_V$ is the closure of the orbit $G'\cdot W$, where
\[  G'=\{\psi\in \GL(E): \psi(V)=V\}.  \]
The stabilizer of $W$ in $\GL(E)$ is
\[  H=\{\psi\in \GL(E): \psi(W)=W\}  \]
so by the isomorphism theorems we obtain:
\[T_W \sigma_V\iso T_eG'/T_e(H\cap G')\iso (T_eG'+T_eH)/ T_eH,\]
\[N_W\sigma_{V}\iso (T_eGL(E)/T_eH)/((T_eG'+T_eH)/ T_eH),\]
and
\[N_W\sigma_{V}\iso\Hom(E,E)/(T_eG'+T_eH).\]
Clearly
\[\begin{split}
 T_eG'+T_eH&= \{\psi\in \Hom(E,E): \psi(V)\subset V\}+\{\psi\in \Hom(E,E): \psi(W)\subset W\} \\
      &=\{\psi\in \Hom(E,E): \psi(W\cap V)\subset \langle W, V\rangle\},
    \end{split}  \]
 and using Lemma \ref{LemAlg}. with $C,D=E$, $C'=W\cap V$ and $D'=\langle W, V\rangle$ we obtain:
\begin{proposition}
\begin{equation*}
\label{NWSV}
N_W\sigma_{V}\cong\Hom\big(W\cap V, E/\langle W, V\rangle\big).
\end{equation*}
\end{proposition}

In fact we get a little more, looking at the proof of Lemma \ref{LemAlg} we see that the canonical quotient map
\[q:T_W\gr_{2a}(E)\to N_W\sigma_{V}\]
can be expressed as
\[ q(\psi)=\pi\circ \psi\circ\iota,\]
where $\pi:E\to E/\langle W, V\rangle$ is the canonical quotient map and $\iota: W\cap V\to E$ is the inclusion.
\subsection{Orienting the normal spaces}
The following fact is well known for those who ever wanted to calculate the orientation of a real Hom-space:
\begin{lemma}
  Suppose that $A,B$ are even dimensional real vector spaces with bases $\{a_i:i=1,\dots,m\}$ and $\{b_j:j=1,\dots,n\}$, respectively. Then the orientation on $\Hom(A,B)$, defined by the lexicographic ordering is independent of the orientations of $A$ and $B$.
\end{lemma}
\begin{proof}
  Interchanging $a_1$ with $a_2$ results in $n$ transpositions in the lexicographic ordering. Similarly, interchanging $b_1$ with $b_2$ results in $m$ transpositions in the lexicographic ordering, either way we add even number of transpositions, so the orientation does not change.
\end{proof}
Notice that the other orientation would also be canonical, we make a choice by using the lexicographic ordering.
\begin{corollary} The vector space $T_W\gr_{2a}(E)$ admit a canonical orientation. The vector spaces $N_W\sigma_{V_i}$ admit canonical orientations if $a$ and $b$ are both even.
\end{corollary}
\subsection{The splitting map $\xi$.} The notion of transversality fitting our purpose the best is
\begin{definition} The subspaces $A_i,\ i=1,\dots n$ of the vector space $A$ are transversal, if the splitting map
\[\xi:=\bigoplus_{i=1}^nq_i:A\to\bigoplus_{i=1}^nA/A_i,  \]
is an isomorphism, where $q_i:A\to A/A_i$ are the canonical quotient maps.
\end{definition}
Moreover we say that submanifolds intersecting at a point are transversal at this point, if their tangent spaces are transversal subspaces of the tangent space of the ambient manifold. For subvarieties we add the extra condition that the intersection point has to be a smooth point of all the subvarieties.

Applying our definition to the varieties $\sigma_{V_i}$ we obtain the splitting map
\[\xi_W:=\bigoplus_{i=1}^4q_i:T_W\gr\to\bigoplus_{i=1}^4N_W\sigma_{V_i}.  \]
Notice that both the source and target of $\xi_W$ are canonically oriented. Assume now that the subvarieties $\sigma_{V_i}$ are transversal. We arrived to the definition of the sign of a balanced subspace:
\begin{definition} We say that the \emph{sign $\varepsilon(W)$ of a balanced subspace} $W$ is $+1$ if $\xi_W$ is orientation preserving, and $-1$ otherwise.
\end{definition}
To determine whether $\xi_W$ is orientation preserving or not, we will pick bases compatible with the given orientations for the vector spaces involved and calculate the determinant of the matrix of $\xi_W$ in these bases. The determinant depends on the choice of the bases but its sign does not.

Let us remind the readers that this definition is simply the unfolding of the cohomological definition of the sign of the cohomology class of the intersection point.
\subsection{The balanced decomposition of $E$ and the determinant of the splitting map $\xi_W$.}
The calculation can be simplified with the proper choice of bases. There is no canonical choice, but almost. Corollary \ref{bijection} immediately implies that
\begin{proposition} The vector space $E$ is the direct sum of minimal (w.r.t. inclusion) balanced subspaces.
\end{proposition}
For the complex case they are all 2-dimensional, corresponding to the eigenvalues of $\varphi$. For the real case the real eigenvalues of $\varphi$ correspond to 2-dimensional, the complex conjugate pairs to 4-dimensional minimal balanced subspaces. We will refer to this decomposition as the \emph{block-decomposition}.  For us the key property is that all the subspaces $V_i$ and all the balanced subspaces are \emph{homogeneous} (in the sense of graded modules) with respect to the block-decomposition: They are the sum of their intersections with the blocks. Let us introduce some notations: The set of minimal balanced subspaces will be denoted by $\mathcal B$. Given a balanced subspace $W$ of dimension $2a$,
\[S_W:=\{B\in \mathcal{B}:B\subset W\},\qquad  V^B_i:=V_i\cap B, \qquad \bar V_i^B:=B/V_i^B. \]
Then we have the decompositions:
\[E=\bigoplus_{B\in \mathcal{B}}B,\qquad W=\bigoplus_{B\in S_W}B, \qquad V_i=\bigoplus_{B\in \mathcal{B}}V_i^B.\]
Consequently, for the tangent and normal spaces:
\[T_W\gr\iso\bigoplus_{A\in S_W}\bigoplus_{B\notin S_W}\Hom(A,B),\qquad N_W\sigma_{V_i}\iso\bigoplus_{A\in S_W}\bigoplus_{B\notin S_W}\Hom(V_i^A,\bar V_i^B).\]
In this decomposition our main object of study, the map $\xi_W$ becomes "block-diagonal" therefore
\[\det(\xi_W)=\prod_{A\in S_W}\prod_{B\notin S_W}\det(\xi^{A,B}),  \]
where $\xi^{A,B}=\bigoplus_{i=1}^4 \xi_i^{A,B}:\Hom(A,B)\to\bigoplus_{i=1}^4\Hom(V_i^A,\bar V_i^B)$ is the splitting map corresponding to the blocks $A$ and $B$, more precisely for $\psi\in \Hom(A,B)$ we have
\begin{equation}\label{eq:pi-i}
  \xi_i^{A,B}(\psi)=\pi\circ \psi\circ\iota,
\end{equation}
where $\pi:B\to \bar V_i^B$ is the canonical quotient map and $\iota: V_i^A\to A$ is the inclusion.

\subsection{Bases for the blocks.} We reduced the main calculation to four cases, depending on the dimension of the blocks $A$ and $B$. It is time to get our hands dirty; in this section we introduce bases for the vector spaces involved.

\subsubsection{Bases for a 2-dimensional block.} According to Corollary \ref{bijection}. a 2-dimensional block $B$ corresponds to a real eigenvalue $\beta$ of $\varphi$. Choose an arbitrary generator $x$ of the one-dimensional space $V_1^B$. Choose the generator of  $V_2^B$ to be $y=\gamma_3(x)$, thus we have the basis $(x,\gamma_3(x))$ of $B$. Then, following the proof of  Corollary \ref{bijection}. $V_3^B$ is spanned by $x+y$ and $V_4^B$ is spanned by $x+\beta y$. Written in the $(x,y)$ basis we have:
\begin{equation}\label{2-dim-bases}
        V_1^B=\left\langle \begin{pmatrix} 1\\0 \end{pmatrix}\right\rangle,
 \qquad V_2^B=\left\langle \begin{pmatrix} 0\\1 \end{pmatrix}\right\rangle,
 \qquad V_3^B=\left\langle \begin{pmatrix} 1\\1 \end{pmatrix}\right\rangle,
 \qquad V_4^B=\left\langle \begin{pmatrix} 1\\ \beta \end{pmatrix}\right\rangle.
\end{equation}
We will also need generators for the spaces $\bar V_i^B$. Using the scalar product defined by $x,y$ we identify $\bar V_i^B$ with the orthogonal complement of $V_i^B$. A convenient choice for generators is:
\begin{equation}\label{2-dim-bar-bases}
       \bar  V_1^B=\left\langle \begin{pmatrix} 0\\-1 \end{pmatrix}\right\rangle,
 \qquad\bar  V_2^B=\left\langle \begin{pmatrix} 1\\0 \end{pmatrix}\right\rangle,
 \qquad\bar  V_3^B=\left\langle \begin{pmatrix} 1\\-1 \end{pmatrix}\right\rangle,
 \qquad\bar  V_4^B=\left\langle \begin{pmatrix} \beta\\-1 \end{pmatrix}\right\rangle.
\end{equation}

\subsubsection{Bases for a 4-dimensional block.} According to Corollary \ref{bijection}. a 4-dimensional block $B$ corresponds to a pair of complex conjugate  eigenvalues $\lambda=\mu+i\nu,\ \bar \lambda=\mu-i\nu$ of $\varphi$. Choose an arbitrary pair of generators $(x,y)$ of the two-dimensional space $V_1^B$. Choose the generator of  $V_2^B$ to be $z=\gamma_3(x),w=\gamma_3(y)$, thus we have the basis $(x,y,\gamma_3(x),\gamma_3(y))$ of $B$. Then, following the proof of  Corollary \ref{bijection}. $V_3^B$ is spanned by $x+z,\ y+w$ and $V_3^B$ is spanned by $x+\mu z-\nu w,\ y+\nu z+\mu w$. In the $x,y,z,w$ bases we get the following, where we omit the brackets and commas for typographical reasons and refer to these matrices as the matrix of the corresponding subspace.):
\begin{equation}\label{4-dim-bases}
        V_1^B=\left\langle \begin{matrix} 1\\0\\0\\0 \end{matrix} \      \begin{matrix} 0\\1\\0\\0 \end{matrix}\right\rangle,
 \ V_2^B=\left\langle \begin{matrix} 0\\0\\1\\0 \end{matrix}      \      \begin{matrix} 0\\0\\0\\1 \end{matrix}\right\rangle,
 \ V_3^B=\left\langle \begin{matrix} 1\\0\\1\\0 \end{matrix}      \      \begin{matrix} 0\\1\\0\\1 \end{matrix}\right\rangle,
 \ V_4^B=\left\langle \begin{matrix} 1\\0\\ \mu\\ -\nu  \end{matrix} \    \begin{matrix} 0\\1\\ \nu\\ \mu \end{matrix}\right\rangle.
\end{equation}
We will also need generators for the spaces $\bar V_i^B$. Using the scalar product defined by $x,y,z,w$ we identify $\bar V_i^B$ with the orthogonal complement of $V_i^B$. A convenient choice for generators is:
\begin{equation}\label{4-dim-bar-bases}
\bar V_1^B=\left\langle \begin{matrix} 0\\0\\1\\0 \end{matrix}     \  \begin{matrix} 0\\0\\0\\1 \end{matrix}\right\rangle,
 \ \bar V_2^B=\left\langle \begin{matrix} 1\\0\\0\\0 \end{matrix}  \          \begin{matrix} 0\\1\\0\\0 \end{matrix}\right\rangle,
 \ \bar V_3^B=\left\langle \begin{matrix} 1\\0\\-1\\0 \end{matrix}  \          \begin{matrix} 0\\1\\0\\-1 \end{matrix}\right\rangle,
 \ \bar V_4^B=\left\langle \begin{matrix} -L\\0\\ \mu\\ -\nu  \end{matrix} \    \begin{matrix} 0\\-L\\ \nu\\ \mu \end{matrix}\right\rangle,
\end{equation}
where $L=\mu^2+\nu^2$.
\subsection{The matrix of the splitting map.} According to (\ref{eq:pi-i}) we need to calculate the matrix of an inclusion and a quotient map. We identified the quotients $\bar V_i^B$ with orthogonal complements. Under this identification the quotient map is identified with an orthogonal projection. We will use the following trivial facts:
\begin{lemma}\label{le:inclusion} Let $C$ be a subspace of the vector space $D$. Let $\{d_i:i=1,\dots,n\}$ be a basis for $D$ and $\{c_j=\sum c_{ji}d_i :j=1,\dots,m\}$ a basis for $C$. Then the matrix $M$ of the inclusion $\iota:C\to D$ in these bases is $M_{ij}=c_{ji}$.
\end{lemma}
\begin{lemma}\label{le:projection}
  Let $S$ be a subspace of the vector space $T$ equipped with a scalar product. Let $\{t_k:k=1,\dots,v\}$ be an orthonormal basis for $T$ and $\{s_l=\sum s_{lk}t_k :l=1,\dots,u\}$ an orthonormal  basis for $S$. Then the matrix $P$ of the orthogonal projection $\pi:T\to S$ in these bases is $P_{lk}=s_{lk}$.
\end{lemma}
The splitting map is a linear map between vector spaces of linear maps. A natural choice of a basis in both the domain and range is the set of linear transformations $e_{ik}$ with matrices $E_{ki}$ with a single 1 entry (we already have bases on our vector spaces!) with the row-column lexicographic ordering.  Combining the previous two lemmas we get:
\begin{lemma}\label{le:inclusion} Let $C$ be a subspace of the vector space $D$. Let $\{d_i:i=1,\dots,n\}$ be a basis for $D$ and $\{c_j=\sum c_{ji}d_i :j=1,\dots,m\}$ a basis for $C$.

Let $S$ be a subspace of the vector space $T$ equipped with a scalar product. Let $\{t_k:k=1,\dots,v\}$ be an orthonormal basis for $T$ and $\{s_l=\sum s_{lk}t_k :l=1,\dots,u\}$ an orthonormal  basis for $S$.

Then the matrix $Q$ of $q_{ik}=\pi\circ  e_{ik}\circ \iota$ is $Q_{lj}= s_{lk}c_{ji}$ for the inclusion $\iota:C\to D$ and orthogonal projection $\pi:T\to S$.
\end{lemma}

Now we can construct the matrix of the splitting map. The bases for $\bar V_i^B$ are not orthonormal, only orthogonal, so we normalize with an appropriate scaling factor. Since this factor is positive (the norm square of the corresponding vector), it does not alter the sign of the determinant, so we omit it. Applying the previous lemma for $C=V_i^A$, $D=A$, $T=B$, $S=V_i^B$, and the basis described in the previous section we can see that we need to take the \emph{tensor products} of the matrices of $V_i^A$ and $ \bar V_i^B$, transpose them and write them under each other.
\subsubsection{$A$ and $B$ are 2-dimensional} Let $A$ correspond to the eigenvalue $\alpha$ and $B$ correspond to the eigenvalue $\beta$.
Following the instructions above we arrive to the matrix
\begin{equation}\label{m22}
  \xi^{\alpha\beta}=\left(
             \begin{array}{cccc}
               0 & 0 & -1 & 0 \\
               0 & 1 & 0 & 0 \\
               1 & 1 & -1 & -1 \\
               \beta & \alpha\beta & -1 & -\alpha \\
             \end{array}
           \right)
\end{equation}
The determinant is easy to calculate by hand:
\begin{equation}\label{dm22}
  \det(\xi^{\alpha\beta})=\beta-\alpha.
\end{equation}
\subsubsection{$A$ is 2-dimensional, $B$ is 4-dimensional} Let $A$ correspond to the eigenvalue $\alpha$ and $B$ correspond to the pair of eigenvalues $\lambda=\mu+i\nu,\ \bar \lambda=\mu-i\nu$.
Following the instructions above we arrive to the matrix:
\begin{equation}\label{m24}
\xi^{\alpha,\mu\pm i\nu}=\left( \begin {array}{cccccccc}
0&0&0&0&1&0&0&0\\ \noalign{\medskip}0
&0&0&0&0&0&1&0\\ \noalign{\medskip}0&1&0&0&0&0&0&0
\\ \noalign{\medskip}0&0&0&1&0&0&0&0\\ \noalign{\medskip}1&1&0&0&-1&-1&
0&0\\ \noalign{\medskip}0&0&1&1&0&0&-1&-1\\ \noalign{\medskip}-L&-\alpha\, L&
0&0&\mu&\alpha\,\mu&\nu&\alpha\,\nu\\ \noalign{\medskip}0&0&-L&-\alpha\,L&\nu&\alpha\,\nu&\mu&\alpha\,\mu\end {array} \right)
\end{equation}
The determinant is:
\begin{equation}\label{dm24}
 \det(\xi^{\alpha,\mu\pm i\nu})=\left( {\nu}^{2}+{\mu}^{2} \right)  \left( {\mu}^{2}-2\,\alpha\,\mu+{
\alpha}^{2}+{\nu}^{2} \right)=(\nu^{2}+{\mu}^{2})(\lambda-\alpha)(\bar\lambda-\alpha).
\end{equation}
\subsubsection{$A$ is 4-dimensional, $B$ is 4-dimensional} Let $A$ correspond to the pair of eigenvalues $\mu+i\nu,\ \mu-i\nu$  and $B$ correspond to the pair of eigenvalues $\delta=\varkappa+i\vartheta,\ \bar\delta=\varkappa-i\vartheta$.
Following the instructions above and using the notation $K=\varkappa^2+\vartheta^2$ we arrive to the matrix of $\xi^{\mu\pm i\nu,\varkappa\pm i\vartheta}$:

\begin{equation}\label{m44} \setlength{\arraycolsep}{2.5pt}
\left( \begin {array}{cccccccccccccccc}
			       0&0&0&0&0&0&0&0&1&0&0&0&0&0&0&0
\\ \noalign{\medskip}0&0&0&0&0&0&0&0&0&1&0&0&0&0&0&0
\\ \noalign{\medskip}0&0&0&0&0&0&0&0&0&0&0&0&1&0&0&0
\\ \noalign{\medskip}0&0&0&0&0&0&0&0&0&0&0&0&0&1&0&0
\\ \noalign{\medskip}0&0&1&0&0&0&0&0&0&0&0&0&0&0&0&0
\\ \noalign{\medskip}0&0&0&1&0&0&0&0&0&0&0&0&0&0&0&0
\\ \noalign{\medskip}0&0&0&0&0&0&1&0&0&0&0&0&0&0&0&0
\\ \noalign{\medskip}0&0&0&0&0&0&0&1&0&0&0&0&0&0&0&0
\\ \noalign{\medskip}1&0&1&0&0&0&0&0&-1&0&-1&0&0&0&0&0
\\ \noalign{\medskip}0&1&0&1&0&0&0&0&0&-1&0&-1&0&0&0&0
\\ \noalign{\medskip}0&0&0&0&1&0&1&0&0&0&0&0&-1&0&-1&0
\\ \noalign{\medskip}0&0&0&0&0&1&0&1&0&0&0&0&0&-1&0&-1
\\ \noalign{\medskip}-K&0 &-\mu\,K&\nu\, K& 0&0&0&0& \varkappa&0&\mu\,\varkappa&-\nu\,\varkappa& -\vartheta&0&-\mu\,\vartheta&\nu\,\vartheta
\\ \noalign{\medskip}0&-K &-\nu\,K&-\mu\, K& 0&0&0&0& 0&\varkappa&\nu\,\varkappa&\mu\,\varkappa& 0&-\vartheta&-\nu\,\vartheta&-\mu\,\vartheta
\\ \noalign{\medskip}0&0&0&0& -K&0 &-\mu\,K&\nu\, K& \vartheta&0&\mu\,\vartheta&-\nu\,\vartheta &\varkappa&0&\mu\,\varkappa&-\nu\,\varkappa
\\ \noalign{\medskip}0&0&0&0& 0&-K &-\nu\,K&-\mu\, K& 0&\vartheta&\nu\,\vartheta&\mu\,\vartheta & 0&\varkappa&\nu\,\varkappa&\mu\,\varkappa
\end {array}
 \right)
\end{equation}
\bigskip

The determinant is:
\begin{equation}\label{dm44}
\det(\xi^{\mu\pm i\nu,\varkappa\pm i\vartheta})= K^2\big((\varkappa-\mu)^2+(\vartheta-\nu)^2\big)\big(\varkappa-\mu)^2+(\vartheta+\nu)^2\big)
=K^2(\delta-\lambda)(\delta-\bar\lambda)(\bar\delta-\lambda)(\bar\delta-\bar\lambda).
\end{equation}
\subsubsection{$A$ is 4-dimensional, $B$ is 2-dimensional} Let $A$ correspond to the pair of eigenvalues $\mu+i\nu,\ \mu-i\nu$  and $B$ correspond to the eigenvalue $\beta$.
Following the instructions above we arrive to the matrix:
\begin{equation}\label{m42}
\xi^{\mu\pm i\nu,\beta}=\left( \begin {array}{cccccccc}
0&0&0&0&-1&0&0&0\\ \noalign{\medskip}0
&0&0&0&0&-1&0&0\\ \noalign{\medskip}0&0&1&0&0&0&0&0
\\ \noalign{\medskip}0&0&0&1&0&0&0&0\\ \noalign{\medskip}1&0&1&0&-1&0
&-1&0\\ \noalign{\medskip}0&1&0&1&0&-1&0&-1\\ \noalign{\medskip}\beta&
0&\beta\,\mu&-\beta\,\nu&-1&0&-\mu&\nu\\ \noalign{\medskip}0&\beta&\beta\,\nu&\beta\,\mu
&0&-1&-\nu&-\mu\end {array} \right)
\end{equation}
The determinant is:
\begin{equation}\label{dm42}
\det(\xi^{\mu\pm i\nu,\beta})=\nu^{2}+{\mu}^{2}-2\,\mu\,\beta+{\beta}^{2}=\nu^{2}+(\beta-\mu)^2=(\beta-\lambda)(\beta-\bar\lambda).
\end{equation}

\subsection{Determining the sign of a balanced subspace} All determinants we calculated are positive, except the first one, so we have the following:
\begin{theorem}\label{thm:sign-alg-real} Let $W$ be a balanced subspace, let $R_W$ denote the set of real eigenvalues of $\varphi$ for which the corresponding eigenspace is in $W$, and $\bar R_W$ denote the set of real eigenvalues of $\varphi$ for which the corresponding eigenspace is not in $W$. Then $W$ has a positive sign if and only if the resultant
\[ \prod_{\alpha\in R_W}\prod_{\beta\in \bar R_W}\beta-\alpha   \]
is positive.
\end{theorem}
Using the complex eigenvalues as well gives the following variant:
\begin{theorem}\label{thm:sign-alg-all} Let $W$ be a balanced subspace, let $E_W$ denote the set of all eigenvalues of $\varphi$ for which the corresponding eigenspace is in $W$ (front eigenvalues), and $\bar E_W$ denote the set of all eigenvalues of $\varphi$ for which the corresponding eigenspace is not in $W$ (back eigenvalues). Then $W$ has a positive sign if and only if the resultant
\[ \prod_{\lambda\in E_W}\prod_{\delta\in \bar E_W}\delta-\lambda   \]
is positive.
\end{theorem}
This latter resultant has a linear algebraic interpretation. $V_1$ can be written as $V_1=C\oplus D$, where $C=V_1\cap W$ and $D$ is the (unique) complementary invariant subspace of $\varphi$. Let $\chi_C:=\det(\varphi|_C-xI)$ and $\chi_D:=\det(\varphi|_D-xI)$ denote the corresponding characteristic polynomials. Then
\[ \prod_{\lambda\in E_W}\prod_{\delta\in \bar E_W}\delta-\lambda=\res(\chi_C,\chi_D),  \]
where $\res\big(f(x),g(x)\big)$ denotes the usual resultant of two polynomials.

\begin{remark} \label{rem:generic}
We can also see that our genericity condition is correct, if the eigenvalues are different from 0 and 1, and distinct then $\det(\xi_W)$ is nonzero for any balanced subspace $W$, therefore the subvarieties $\sigma_{V_i}$ are transversal. Also, it is not difficult to see that this is the maximal possible genericity condition: if the subvarieties $\sigma_{V_i}$ are not transversal, then at least one of the intersection points will have a multiplicity, forcing the solution function to drop its value or having infinity as its value, so it will not be locally constant at this configuration.
\end{remark}

\begin{remark}\label{rem:odd} We could use Theorem \ref{thm:sign-alg-real}. as the definition of the sign of a solution, and this definition would make sense for $a$ or $b$ odd as well. However these signs would depend on the order of the subspaces $V_i$ and lack the geometric meaning.
\end{remark}
\section{Combinatorics: the signed sum of solutions and the number of solutions modulo 4}\label{sec:mod4}
In this chapter we prove by a combinatorial argument that the bounds obtained in Remark \ref{bounds}.\ are in fact the signed sums of the solutions (defining  the sign $\varepsilon(W)$ as the the sign of $\det(\xi_W)$, see Remark \ref{rem:odd}.). In Section \ref{sec:schubert} for $a$ and $b$  even we sketched a cohomological proof of this equality, based on real Schubert calculus. The result suggests that a cohomological argument should exist for the other cases as well. Along the proof we get an other interesting result: for $a$ and $b$  even the number of solutions is constant modulo 4. It was observed in \cite{sottile4} that for a wide class of enumerative questions the number of solutions is constant modulo 4. They study so called osculating solutions. It can be considered as an other ``real form'' for the complex Schubert problems, so we can expect similar behavior. Notice that in the  example of lines on a smooth cubic (calculated by B. Segre in \cite{segre}) it is also true that the number of solutions is constant modulo 4. It would be interesting to find a universal explanation. We also show that if $a$ and $b$ are both divisible by $2^k$, then the number of solutions is constant modulo $2^{k+1}$.

\begin{definition} \label{def:sol} Let $a,b,c$ be positive integers satisfying the inequalities $a\leq b$ and $2c\leq b$. Let $B_i$ denote the set $\{b-2i+1, b-2i+2\}$ for $0<i\leq v:=\lfloor b/2\rfloor$. We will call $B_i$ the $i$-th block. Then we use the notation
 \[\sol(a,b,c):=\big\{H\subset \{1,\dots,b\}\ :\ |H|=a \text{ and}\ |H\cap B_i|\neq1\ \text{ for}\ i=1,\dots,c\big\},  \]
 and call it the set of solutions.
\end{definition}
Notice that we indexed the blocks backwards for notational convenience. Elements of $\{1,\dots,b\}$ correspond to the eigenvalues of $\varphi$, among which the first $b-2c$ are real. The remaining eigenvalues are arranged into complex conjugate pairs, corresponding to the blocks $B_1,\dots,B_c$

Corollary \ref{bijection}.\ implies that the sets $H\in \sol(a,b,c)$ are in bijection with the $2a$-dimensional balanced subspaces $W_H$ for a 4-space configuration with $c$ complex conjugate eigenvalues.

For $H\in \sol(a,b,c)$ we assign a permutation $\rho_H\in S_{b-2c}$ by listing the elements of the set $\{1,\dots,b-2c\}\cap H$ in increasing order first, then listing the elements of the set $\{1,\dots,b-2c\}\setminus H$ in increasing order.
\begin{definition} \label{def:sign-set} The sign $\varepsilon(H)$ of a set $H\in \sol(a,b,c)$ is the sign of the permutation $\rho_H\in S_{b-2c}$.
\end{definition}
Theorem \ref{thm:sign-alg-real}. implies that $\varepsilon(H)=\varepsilon(W_H)$.

We introduce three easy lemmas leading to the results of the section. We keep $a$ and $b$ fixed in the rest of the section, so we can omit them from the notation.
\begin{lemma}\label{lem:cs-partition} Let $S$ be a subset of $\{1,\dots,v\}$ for $v=\lfloor b/2\rfloor$. With the notation
\[ C_S=\big\{H\subset \{1,\dots,b\}\ :\ |H|=a \text{ and}\ |H\cap B_i|=1\ \Leftrightarrow i\in S\big\},  \]
we obtain a partition of the set of solutions:
\[\sol(a,b,c)=\bigcup\big\{C_S:\ S\subset \{c+1,\dots,v\}\big\}.\]
\end{lemma}
In other words we sort the solutions according to which blocks they intersect properly.
\begin{lemma}\label{lem:mod4} For any  $S\subset\{1,\dots,v\}$ we have
\[ \twocases{|C_S|=} {2^{|S|} \cdot \ds\binom{v-|S|}{\left\lfloor\frac{a-|S|}2\right\rfloor}}
                      {\text{\rm $b$ is odd or $a-|S|\geq 0$  is even,}}
                       {0}{\text{\rm otherwise.}}{}   \]
\end{lemma}

\begin{lemma}\label{lem:signed-sum} For any non-empty $S\subset\{1,\dots,v\}$ we have $\displaystyle\sum\limits_{H\in C_S}\varepsilon(H)=0$.
\end{lemma}
Lemma \ref{lem:mod4}. implies that
\[  N_{\R}(a,b,c)\equiv |C_\emptyset|+\sum_{i=c+1}^v|C_{\{i\}}|   \]
modulo 4, therefore 
\begin{theorem}\label{thm:mod4}\mbox{}

 \begin{enumerate}
  \item If $a$ or $b$ is odd then $N_{\R}(a,b,c)=|\sol(a,b,c)|$ is congruent to $|C_\emptyset|$ modulo 4 independently of $c$ if and only if
  \[ \frac12|C_{\{i\}}|=\ds\binom{v-1}{\left\lfloor\frac{a-1}2\right\rfloor} \]
is even.
  \item For $a=2u$ and $b=2v$  even, the number of solutions $N_{\R}(a,b,c)=|\sol(a,b,c)|$ is congruent to $|C_\emptyset|=\binom{v}{u}$ modulo 4 independently of $c$.

\end{enumerate}
\end{theorem}
In fact for $a=2u$ and $b=2v$  even we can even give a formula,
\begin{equation}\label{eq:4adi}
N_{\R}(a,b,c)=\binom{v}{u}+\sum_{i=1}^v 4^i\binom{v-c}{2i}\binom{v-2i}{u-i},
\end{equation}
clearly implying the result in the even case.
For the non even case the smallest example satisfying the modulo 4 property is $b=6$ and $a=3$, so $|C_{\{i\}}|=\binom21$ and the number of solutions for $c=0,1,2,3$ are $20,4,4,0$, respectively.

Theorem \ref{thm:mod4} can be strengthened if further divisibility of  $a$ and $b$ is assumed:
 \begin{theorem}\label{thm:mod2toi} Suppose that $a$ and $b$ are both divisible by $2^k$. Then the number of solutions $N_{\R}(a,b,c)=|\sol(a,b,c)|$ is congruent to $|C_\emptyset|=\binom{v}{u}$ modulo $2^{k+1}$ independently of $c$.
 \end{theorem}
 \begin{proof}
 From equation (\ref{eq:4adi}) we can see that the theorem follows from the divisibility of $\binom{v-2i}{u-i}$ by $2^{k+1-2i}$ under the conditions of the theorem and assuming $2i\leq k$. We use a theorem of Kummer \cite{kummer} on the divisibility of binomial coefficients:
 \begin{theorem}(Kummer) Given integers $n \geq m \geq 0$ and a prime number $p$, the maximum integer $k$ such that $p^k$ divides the binomial coefficient $\tbinom n m$ is equal to the number of carries when $m$ is added to $n - m$ in base $p$.
 \end{theorem}
 We use Kummer's theorem  for $n=v-2i$ and $m=u-i$. Since $2^{k-1}$ divides $u$ and $v-u$, $v-u-i$ and $u-i$ are the same modulo $2^{k-1}$. Therefore the number of 1's in the last $k-1$ digits gives a lower bound on the number of carries in base 2 when adding $v-u-i$ to $u-i$. The number of 1's in the last $k-1$ digits is clearly at least $k-1-\lfloor\log_2 i\rfloor$ in both $u-i$ and $v-u-i$. Finally to show $k-1-\lfloor\log_2 i\rfloor\geq k+1-2i$, we use the inequality $\lfloor\log_2i\rfloor+1\leq i$ ($i=1,2,\ldots$).
 \end{proof}

On the other hand Lemma \ref{lem:cs-partition}. and \ref{lem:signed-sum} immediately implies that
\begin{theorem}\label{thm:signed-sum} For $u=\lfloor a/2\rfloor$ and $v=\lfloor b/2\rfloor$ the signed sum of the solutions
\[ \twocases{\sum_{H\in \sol(a,b,c)}\varepsilon(H)=}{\ 0}{\text{\rm $a$ odd and $b$ even}}{\ds\binom{v}{u}}{\text{otherwise,}}{}  \]
agreeing with the minimal number of solutions as promised.
\end{theorem}

\section{Chambers} \label{sec:chambers}
 As we saw, we have a good understanding of the solution function without a complete list of the chambers, but for the sake of completeness we give a description. Four-tuples of subspaces were studied frequently in the last 70 years, especially subspaces of middle dimension, called \emph{medials} by Turnbull in \cite{turnbull_projective_1942}. Most of the calculations in the section appeared in some versions before.

Let us denote the subspace of generic (in the sense of Definition \ref{de:generic}) four-tuples of spaces in $\Gr_b(\R^{2b})^4$ by $GC$ (as generic configurations). Notice that the group $\GL(2b)$ acts on $GC$. For every generic configuration $V_1,\dots,V_4$ we assigned a linear map $\varphi:V_1\to V_1$. By taking the characteristic polynomial $\chi_\varphi(\lambda):=\det(\lambda I-\varphi)$ we can define a map $\chi:GC\to P$, where $P\iso \R^b$ is the space of degree $b$ polynomials with leading coefficient 1.
\begin{proposition} The fibers of $\chi$ are exactly the orbits of $\GL(2b)$ acting on $GC$. \end{proposition}
\begin{proof} Let $V_1,\dots,V_4$ and $V_1^{'},\dots,V_4^{'}$ be two configurations in $GC$. If $\varphi:V_1\to V_1$ and $\varphi':V^{'}_1\to V^{'}_1$ has the same characteristic polynomial, then there is a linear map $A_1:V_1\to V^{'}_1$ such that $A_1\varphi= \varphi'A_1$ (recall that all eigenvalues of $\varphi$ are distinct). We define $A_2:V_2\to V^{'}_2$ as $A_2:=\gamma_3^{'}A\gamma_3^{-1}$, using the notation of Section \ref{sec:graph}. Then it is elementary to check that the linear map $ A(v_1+ v_2):= A_1(v_1)+A_2(v_2)$  maps the configuration $V_1,\dots,V_4$ to $V_1^{'},\dots,V_4^{'}$.

On the other hand for $g \in \GL(2b)$ the linear map assigned to the configuration $gV_1,\dots,gV_4$ is $g|_{V_1}\varphi (g|_{V_1})^{-1}$, so they have the same characteristic polynomial.
\end{proof}
It is also elementary to check the following:
\begin{proposition} The stabilizer subgroup $H$ of $(V_1,\dots,V_4)\in GC$ is isomorphic to
\[  \{\tau\in\GL(V_1): \tau\varphi=\varphi\tau\}\iso\GL(1,\C)^c\times\GL(1,\R)^{b-2c},\]
 where $c$ is the number of complex conjugate pairs of eigenvalues of $\varphi$ as before. $H$ is the image of $\rho(\tau):=\smx{\tau}{0}{0}{\gamma_3\tau\gamma_3^{-1}}$.
\end{proposition}
Notice that the determinant of $\rho(\tau)$ is always positive, consequently all fibers have 2 components.

In the space of characteristic polynomials $P$ the degeneracy conditions are easy to describe: according to Definition \ref{de:generic} all eigenvalues have to be different and 0 and 1 are not allowed. Therefore connected components of $\im \chi$ are labelled by quadruples of non-negative integers $(c,x,y,z)$, such that $2c+x+y+z=b$ ; $x$ denoting the number of real eigenvalues in $(-\infty,0)$, $y$ the number of them in $(0,1)$ and $z$ the number in $(1,\infty)$. The preimage of such a component  either splits into two chambers or it is connected. We claim that it always splits. The first three subspaces determine an orientation of $E$: pick a basis $e_1,\dots e_b$ in $V_1$, then $e_1,\dots e_b,\gamma_3(e_1),\dots \gamma_3(e_b)$ is a basis of $E$ and it is easy to see that different choices for a basis in $V_1$ define this way the same orientation on $E$. A routine calculation shows that for $g\in \GL(2b)$ the three subspaces $gV_1,gV_2,gV_3$ define the same orientation as the $V_i$'s if and only if $\det(g)>0$. This implies that if $\det(g)<0$ then a generic configuration and its $g$-image are always in different chambers, implying our claim. Summing up:
\begin{proposition}The connected components of the space of generic configurations can be labelled by quadruples of non-negative integers $(c,x,y,z)$, such that $2c+x+y+z=b$, and a $\pm$, depending on the orientation defined by the first three subspaces.
\end{proposition}

\section{Connection with quivers}\label{sec:quiver}
This section is not necessary for our main result, but we believe it helps to connect with other notions the reader might be familiar with. We assume some familiarity with the theory of quivers (see e.g. \cite{ringel_rational_1980}).

Four-space configurations are naturally linked to $D_4$-quivers: Let $E$ and $E_i,\ i=1,2,3,4$ be vector spaces and let
\[ Q=\bigoplus_{i=1}^4\Hom(E_i,E)  \]
denote the corresponding (inwardly oriented) $D_4$-quiver representation space. If $q=(q_1,q_2,q_3,q_4)\in Q$, then $V_i:=\im q_i ,\ i=1,2,3,4$ defines a configuration of four subspaces in $E$. The configuration is generic in the sense of Definition \ref{de:generic}. if and only if the corresponding module $M_q$ splits into non isomorphic indecomposables such that
\begin{itemize}
  \item all occuring indecomposable have dimension vector $(2,1,1,1,1)$ if the base field is the complex numbers,
  \item all occuring indecomposable have dimension vector $(2,1,1,1,1)$ or $(4,2,2,2,2)$ if the base field is the real numbers.
\end{itemize}
We call these quiver representations $q$ \emph{generic}.
These indecomposables of course correspond to our minimal balanced subspaces. The group $\GL(Q):=\GL(E)\times \bigtimes_{i=1}^4\GL(E_i)$ acts on $Q$ and it is easy to see that orbits of generic quiver representations correspond bijectively to linear equivalence classes of generic four-space configurations.

We can identify $2a$-dimensional balanced subspaces with submodules with  dimension vector $(2a,a,a,a,a)$.
The connection can be developed further as the characteristic polynomial map $\chi$ of the previous section can be identified with the quotient map for the geometric quotient of the quiver representation space.
\section{Cross ratios: the geometric meaning of the signs}\label{sec:cross}
If $b=1$ then $\varphi(v)=\alpha v$ for the single eigenvalue $\alpha$. It is easy to see (\cite{kaplenko-ponomarev}) that $\alpha$ is the cross ratio of the four points given by $[V_i]\in\R\!\P^1=\P(E)$. (Or $1/\alpha$ depending on the convention for cross ratio.) Therefore for general $b$ the real eigenvalues of $\varphi$ can be identified with cross ratios in $\R\!\P^1\iso S^1$ corresponding to 2-dimensional balanced subspaces and their intersection with the subspaces $V_i$. If a balanced subspace $W$ is given, then these points in $\R\!\P^1$ are divided into "front" and "back" points (see Theorem \ref{thm:sign-alg-all}). The sign of $W$ is positive if and only if the number of transpositions needed to "separate" the front and back points is even.

Applying a permutation $\rho\in S_4$ to the  the $V_i$'s we change the cross ratios by a Moebius transformation (This is the so-called anharmonic representation of $S_4$.) It is easy to see that such (continuous) transformation of $\R\!\P^1$ does not change the sign described above. This shows that the sign of a solution can be defined for \emph{unordered} four-tuples of subspaces as well. Here we will have less chambers, as the $S_4$-action glues together some of them.

\section{What is the solution of a real enumerative problem?}\label{sec:what}
Given a real enumerative problem ideally we would like to calculate the following data:
\begin{enumerate}
  \item Description of the degenerate configurations, i.e. determining the "walls" of the chambers.
  \item Description of the chambers.
  \item The values of the solution function on the chambers.
  \item Cohomological interpretation.
  \item Calculation of the signs of the solutions and combinatorial or geometric description.
  \item Number theoretic properties of the solution function.
\end{enumerate}
We don't expect to answer all these even for Schubert problems. Sometimes finding all chambers doesn't help much in understanding the solution function. Sometimes it is not possible to give a cohomological interpretation and meaningful definition for signs. But at least we would like to see a sharp lower and upper bound for the number of solutions.

We know one other example where the above program was carried out, the case of lines on a smooth cubic surface. The complex case can be solved with cohomological methods. For the convenience of the reader we quickly review it: A cubic polynomial $q$ defines a section $\sigma_q$ of the bundle $\sym^3(S^*)$ over the space of projective lines $\gr_2(\C^4)$ by restriction ($S$ denotes the tautological subbundle over the Grassmannian). Lines on the cubic surface defined by $q$ correspond to zeroes of $\sigma_q$, so for generic $q$ the number can be calculated from the Euler class of $\sym^3(S^*)$. The integral is easy to calculate:
\[\int_{\gr_2(\C^4)}e\big(\sym^3(S^*)\big)=27. \]
The real case was treated by B. Segre in \cite{segre} who showed that the number of lines is $3,7,15$ or $27$. He also distinguished two types of lines: take a point of a line on the cubic surface. The intersection of the cubic surface and (projective) tangent plane at this point will be a plane cubic curve, one component of which is the line itself. The other component is either an ellipse or a hyperbola. Segre called these lines \emph{elliptic} and \emph{hyperbolic}, respectively.

The cohomological calculation can be carried out in the real case as well (see e.g. \cite{okonek-teleman}):
\[\int_{\gr_2(\R^4)}e\big(\sym^3(S^*)\big)=3, \]
which explains the minimum of the solution function. Signs are defined as the sign of the intersection of the section $\sigma_q$ and the zero section. This was calculated in \cite{okonek-teleman}. The different signs were identified with the elliptic and hyperbolic types in \cite{finashin_abundance_2012}. An explanation for the modulo 4 property was given in \cite{benedetti_spin_1995}.

\bibliography{signs}
\bibliographystyle{alpha}
\end{document}